\documentclass[12pt, a4paper, reqno]{amsproc}
\usepackage{amsmath, amsthm, enumerate}
\usepackage{hyperref}
\newtheorem{theorem}{Theorem}
\newtheorem{lemma}{Lemma}
\newtheorem{proposition}{Proposition}
\newtheorem{example}{Example}

\newtheorem{remark}{Remark}
\DeclareMathOperator{\logdens}{log\ dens} 
\author[D.  Kumar, S. Kumar and M. Saini]{ Dinesh Kumar, Sanjay Kumar and Manisha Saini}
\address{department of mathematics, deen dayal upadhyaya college, university of delhi, new delhi-110078, india.}
\email{dinukumar680@gmail.com}
\address{department of mathematics, deen dayal upadhyaya college, university of delhi, new delhi-110078, india.}
\email{skpant@ddu.du.ac.in}
\address{department of mathematics, university of delhi, delhi-110007, india.}
\email{sainimanisha210@gmail.com, msaini@maths.du.ac.in}
\thanks {The third author is Senior Research Fellow (UGC, Delhi).}
\keywords {entire function, meromorphic function, order of growth, complex differential equation}
\subjclass[2010]{34M10, 30D35}

\title[Solutions of Higher Order]{Solutions of Higher Order Linear Differential Equations}
\begin{document}
\setlength{\abovedisplayskip}{4pt}
\setlength{\belowdisplayskip}{4pt}
\maketitle
\begin{abstract}
We show that the higher order linear differential equation possesses all solutions of infinite order under certain conditions by extending the work of authors about second order differential equation \cite{dsm2}.
\end{abstract}
\section{Introduction}
For entire functions $A_{m-1}(z), \ldots, A_0(z)$ and $H(z)$, the differential equation 

\begin{equation}\label{sde3}
f^{(m)}+A_{m-1}(z)f^{(m-1)}+\ldots+A_0(z)f=H(z), m\geq3
\end{equation}
has entire functions as its solutions, where $A_0(z), H(z)\not\equiv 0$. If functions $A_{m-1}(z),\ldots, A_0(z)$ are polynomials and $H(z)$ is an entire function of finite order then all solutions of equation \eqref{sde3} have finite order. Therefore, if at least one of the coefficients is transcendental entire then a solution of infinite order of equation \eqref{sde3} exists. The associated homogeneous linear differential equation 
\begin{equation}\label{sde2}
f^{(m)}+A_{m-1}(z)f^{(m-1)}+\ldots+A_0(z)f=0
\end{equation}
has all non-trivial solutions of finite order if and only if all coefficients are polynomials \cite{lainebook}. It is well known that a solution of equation \eqref{sde3} is related to solution of equation \eqref{sde2}. The aim of this article is to find a necessary condition for the non-existence of solutions of finite order of equation \eqref{sde3}. Wang and Laine \cite{wanglaine}  proved that solutions of equation \eqref{sde3} are of infinite order when orders of coefficients $A_{m-1}(z), \ldots, A_0(z)$ are all equal. The authors have established certain conditions under which the associated homogeneous differential equation of \eqref{sde3} possesses all solutions having infinite order \cite{sm}. The main result of this paper is a generalization of Theorem 2 in \cite{dsm2} to higher order linear differential equations which we state below. We follow the notations $\rho(f),$ $\lambda(f)$ and $\rho_2(f)$ for order of growth, exponent of convergence and hyper-order of growth of entire function $f$ respectively, as used in \cite{dsm,sm,sm2,dsm2}.
\begin{theorem}\label{higthm}
Suppose that there exists a fixed integer $j\in \{1,2,\ldots, m-1\}$ such that $\lambda(A_j)<\rho(A_j)$, $A_0(z)$ is a transcendental entire function satisfying $\rho(A_0)\neq \rho(A_j)$ and $\max\{\rho(A_k):k=1,2,\ldots,m-1, k\neq j\}<\rho(A_0)$. Also, suppose that $H(z)$ is an entire function with $\rho(H)<\max\{\rho(A_0),\rho(A_j)\}$. Then all transcendental solutions $f$ of equation \eqref{sde3} satisfies
\begin{enumerate}[(a)]
\item $\rho(f)=\infty$
\item $\lambda(f)=\infty$
\item $\rho_2(f)= \max\{\rho(A_0),\rho(A_j)\}$ where $\max\{\rho(A_0),\rho(A_j)\}$ is a finite quantity.
\item For every $c\in \mathbb{C}$,  $\delta(c,f)=0$ and therefore, $f$ has no finite deficient value.
\end{enumerate}
\end{theorem}
\begin{remark}
Under the hypothesis of Theorem \ref{higthm}, equation \eqref{sde3} may possesses non-constant polynomial solutions. Also, by order consideration of an entire function, we obtain that all non-constant polynomial solutions of equation \eqref{sde3} are of degree less than $j$, where $j\in \{1,2, \ldots,m-1\}$ is fixed in Theorem \ref{higthm}. However, when $j=1$, then equation \eqref{sde3} has no polynomial solution . 
 \end{remark}
 
 The following examples justify that the conditions in the hypothesis of Theorem \ref{higthm}(a) cannot be relaxed.

\begin{example}\label{higeg1}
The finite order function $f(z)=e^{-z}$ satisfies the linear differential equation
\begin{equation*}
f'''+e^zf''-f'-(e^z-1)f=e^{-z}
\end{equation*}
\end{example}
Here we have, $\rho(A_k)<\rho(A_0)=\rho(A_j)$ and $\rho(H)=\max\{ \rho(A_0),\rho(A_j)\}$ for $k=1$ and $j=2$, which shows that hypothesis in Theorem \ref{higthm} are necessary.
\begin{example}\label{higeg2}
The differential equation 
\begin{equation*}
f'''-2zf''+e^zf'-(2ze^z-1)f=(8z+1)e^{z^2}
\end{equation*}
is satisfied by the finite order function $f(z)=e^{z^2}$.
\end{example}
Here we have, $\rho(A_k)<\rho(A_0)=\rho(A_j)$ and $\rho(H)>\max\{\rho(A_0),\rho(A_j)\}$ for $k=2$ and $j=1$, which also implies that hypothesis of Theorem \ref{higthm} are necessary.
\begin{example}\label{higeg3}
The linear differential equation
\begin{equation*}
f^{(iv)}+f'''-e^zf''-f'+(2e^z-1)f=1
\end{equation*}
has a finite order solution $f(z)=e^{-z}$, where $\rho(A_k)<\rho(A_0)=\rho(A_j)$ and $\rho(H)<\max\{\rho(A_0),\rho(A_j)\}$ for $k=1,3$ and $j=2$.
\end{example}
 
\begin{example}\label{higeg4}
The finite order function $f(z)=e^{-z}$ is a solution of linear differential equation
\begin{equation*}
f'''+(e^z-1)f''+e^{z^2}f'+e^zf=1-e^{z^2-z}
\end{equation*}
where $\rho(A_j)\neq \rho(A_0)=\rho(A_k)$ and $\rho(H)=\max\{\ \rho(A_0),\rho(A_j)\}$  for $k=2$ and $j=1$.
\end{example}
\begin{example}\label{higeg5}
The differential equation
\begin{equation*}
f'''+(e^{z^2}+1)f''-e^zf'-(e^{z^2}-e^z)f=2
\end{equation*}
has a finite order solution $f(z)=e^{-z},$ where $\rho(A_j)\neq \rho(A_0)=\rho(A_k)$ and $\rho(H)<\max\{\rho(A_0),\rho(A_j)\}$ for $k=2$ and $j=1$.
\end{example}
\begin{example}\label{higeg6}
The differential equation
\begin{equation*}
f'''+f''+e^zf'+\cos{z^{\frac{1}{2}}}f=(6+6z+3z^2e^z+z^3\cos{z^{\frac{1}{2}}})
\end{equation*}
is satisfied by the polynomial $f(z)=z^3$ and $\rho(A_k)<\rho(A_0)$ for $k=2$ and  $\rho(A_0)\neq \rho(A_j)$ and $\rho(H)=\max\{\rho(A_0),\rho(A_j)\}$ for $j=1.$
\end{example}
\begin{example}\label{higeg7}
The function $f(z)=e^{z^2}$ is a finite order solution of the differential equation
\begin{equation*}
f'''+e^{-z}f''+f'-(4z^2+2)e^{-z}f=(14z+8z^3)e^{z^2}
\end{equation*}
where $\rho(A_k)<\rho(A_0)=\rho(A_j)$ and $\rho(H)>\max\{\rho(A_0),\rho(A_j)\}$ for $k=1$ and $j=2$.
\end{example}
\section{Auxiliary Results}
This section is devoted to the known results which will be useful in proving the main theorem. For a subset $E\subset(1,\infty)$, $m(E), m_l(E), \overline{\logdens}(E)$ and $\underline{\logdens}(E)$ denotes the linear measure, logarithmic measure, upper logarithmic density and lower logarithmic density respectively. 

The following lemma of Gundersen \cite{log gg} provides estimates for a meromorphic function outside a set of finite logarithmic measure.
\begin{lemma}\label{gunlem}
Let $f$ be a meromorphic function and let $\Gamma=\{(k_1,j_1),\ldots, (k_p,j_p)\}$ be the set of distinct pairs of integers such that $k_t>j_t\geq0$ for $t=1,2,\ldots, p$. Let $\alpha>1$ and $\epsilon>0$ be given real constants. Then there exists $E\subset(1,\infty)$ satisfying $m_l(E)<\infty$ and a constant $c>0$ depending on $\alpha$ and $\Gamma$ such that 
\begin{equation} 
\left| \frac{f^{(k)}(z)}{f^{(j)}(z)}\right|\leq c \left( \frac{T(\alpha r,f)}{r} \log^{\alpha}{r} \log{T(\alpha r,f)} \right)^{(k-j)}
\end{equation}
Moreover, if $f(z)$ is of finite order then $f(z)$ satisfies:
\begin{equation} \label{guneq1}
\left|\frac{f^{(k)}(z)}{f^{(j)}(z)}\right| \leq |z|^{(k-j)(\rho(f)-1+\epsilon)}
\end{equation} 
 for all $z$ satisfying $|z|\notin E\cup [0,1]$ and $|z| \geq R_0$ and for all $(k,j)\in \Gamma$.
\end{lemma}

The next lemma is used to establish estimates for a transcendental entire function.
\begin{lemma}\label{implem}  \cite{banklang}
Let $A(z)=v(z)e^{P(z)}$ be an entire function, where $P(z)$ is a polynomial of degree $n$ and $v(z)$ is an entire function of order less than $n$. Then for every $\epsilon>0$ there exists $E \subset [0,2\pi)$ of linear measure zero such that

\begin{enumerate}[(i)]

\item for $ \theta \in [0,2\pi) $ with $\delta(P,\theta)>0$, there exists $ R>1 $ satisfying
\begin{equation}\label{eqA1}
 \exp \{ (1-\epsilon) \delta(P,\theta)r^n \} \leq|A(re^{\iota \theta})| \leq  \exp \{ (1+\epsilon) \delta(P,\theta)r^n \}
\end{equation}
for $r>R;$

\item for $\theta \in [0,2\pi)$ with $\delta(P,\theta)<0$, there exists $R>1$ satisfying 
\begin{equation}\label{eq2le}
 \exp \{(1+\epsilon) \delta(P,\theta)r^n \} \leq |A(re^{\iota \theta})| \leq \exp \{ (1-\epsilon)\delta(P,\theta) r^n \} 
\end{equation}
for $r>R.$
\end{enumerate}
\end{lemma}
The following lemma gives upper bound for solutions of equation \eqref{sde2}.
\begin{lemma}\cite{ber}\label{asslem}
Suppose that $\rho(A_k)\leq \rho'<\infty$ for all $k=0,1,\ldots, m-1$. If $f$ is a solution of equation \eqref{sde2} then $\rho_2(f)\leq\rho'$.
\end{lemma}
The next lemma provides a lower bound for modulus of an entire function in a neighborhood of a particular $\theta\in [0,2\pi).$
\begin{lemma}\label{wale}\cite{junwang} 
Suppose $f(z)$ is an entire function of finite order $\rho(f)$ and $M(r,f)=|f(re^{\iota \theta_r})|$ for every $r$. Given $\zeta>0$ and $0<C(\rho(f),\zeta)<1$ there exists $0<l_0<1/2$ and a set $S\subset (1,\infty)$ with $\underline{\log dens}(S)\geq1-\zeta$ such that 
\begin{equation}
e^{-5\pi}M(r,f)^{1-C}\leq |f(re^{\iota \theta})|
\end{equation}
for all sufficiently large $r\in S$  and for all $\theta$ satisfying $|\theta-\theta_r|\leq l_0$.
\end{lemma}
The following result is from \cite{lainebook} and includes the central index of an entire function.
\begin{lemma}\label{lemm}
Let $f$ be a transcendental entire function, $\delta\in (0,1/4)$ and $z$ be such that $|z|=r$ and that
\begin{equation}\label{prelim1}
|f(z)|>M(r,f)\nu(r,f)^{-\frac{1}{4}+\delta}
\end{equation}
holds. Then there exists a set $F\subset (1,\infty) $ with $m_l(F)<\infty$ such that
\begin{equation*}
f^{(p)}(z)=\left( \frac{v(r,f)}{z}\right)^p(1+o(1))f(z)
\end{equation*}
holds for all non-negative integers $p$ and for all $r\notin F$.
\end{lemma}
\begin{remark}
If $|f(re^{\iota \theta})|=M(r,f)$ then equation \eqref{prelim1} holds and there exists $F\subset(1,\infty)$ with $m_l(F)<\infty$ such that
\begin{equation}
\left|\frac{f^{(p)}(re^{\iota \theta})}{f(re^{\iota \theta})}\right|=\left(\frac{v(r,f)}{r}\right)^p(1+o(1))
\end{equation}
for all non-negative integers $p$ and for all $r\notin F$. We know that the central index of a transcendental entire function $f$ satisfies $\nu(r,f)\geq1$, as a result we have
\begin{equation}\label{prelim2}
\left|\frac{f^{(p)}(re^{\iota \theta})}{f(re^{\iota \theta})}\right|\geq \frac{1}{r^p}(1+o(1))
\end{equation}
holds for all non-negative integers $p$ and $r\notin F$.
\end{remark}
The following result \cite[Proposition 9.3.2]{lainebook} provides relation between the proximity function of $1/(f-c)$ and characteristic function of $f$.
\begin{proposition}\label{mohon'ko}
Let $P(z,f)$ be a polynomial in $f$ and its derivatives with meromorphic coefficients $a_{\kappa}$, $\kappa \in I$. Suppose that $f$ is a transcendental meromorphic function solution of $P(z,f)=0$ and $c$ is a complex number. If $P(z,c)\neq 0$ then 
\begin{equation}
m\left(r,\frac{1}{f-c}\right)=O\left( \sum_{\kappa \in I} T(r,a_{\kappa}) \right)+S(r,f).
\end{equation}

\end{proposition}

The next three lemmas provides relation between maximum modulus and characteristic functions of two entire functions under certain conditions.
\begin{lemma}\label{ordlem}\cite{dsm2}
Suppose  $f(z)$ is an entire function with $\rho(f)\in(0,\infty)$. Then for each $\epsilon>0,$ there exists a set $S\subset (1,\infty)$ that satisfies $\underline{\log
dens}(S)>0$ and
\begin{equation}\label{meroine}
M(r,f)\geq \exp\{r^{\rho(f)-\epsilon}\}
\end{equation}
for all $r$ sufficiently large and $r\in S$.
\end{lemma}
\begin{lemma}\label{olem}\cite{dsm2}
Let $f(z)$ and $g(z)$ be two meromorphic functions satisfying $\rho(g)<\rho(f).$ Then there exists a set $S\subset(1,\infty)$ with $\underline{\log dens}(S)>0$ such that  
$$T(r,g)=o(T(r,f))$$
 for sufficiently large $r\in S$.
\end{lemma}
\begin{lemma}\label{lessord}\cite{dsm2}
Suppose  $f(z)$ and $g(z)$ be two entire functions satisfying $\rho(g)<\rho(f)$. Then for $0<\epsilon\leq \min\{3\rho(f)/4, \left(\rho(f)-\rho(g)\right)/2 \}, $ there exists $S\subset (1,\infty)$ with $\overline{\log dens}(S)=1$ satisfying
$$ |g(z)|=o\left(M(|z|,f)\right)$$
for sufficiently large $|z|\in S$.
\end{lemma}
\section{Proof of Main Theorem}
We state and prove a lemma which will be used in the proof of Theorem \ref{higthm}.
\begin{lemma}\label{probig}
Suppose  $A_{m-1}(z),\ldots, A_0(z)$ and $H(z)$ are entire functions and there is an integer $j\in\{1,2, \ldots, m-1\}$ such that $\rho(A_j)\neq\rho(A_0)$, $\max\{\rho(A_k):k=1,2,\ldots, m-1, k\neq j\}$ and $\rho(H)<\max\{ \rho(A_j), \rho(A_0)\}$. Then all transcendental solutions $f$ of  equation \eqref{sde3} of finite order satisfies $\rho(f)\geq \max\{ \rho(A_j),\rho(A_0)\}.$
\end{lemma}
\begin{proof}
 Suppose that $\rho(A_j)<\rho(A_0)$. Then using equation \eqref{sde3}, first fundamental theorem of Nevanlinna theory, lemma of logarithmic derivatives and Lemma \ref{olem} we have
\begin{align*}
m(r,A_0)&\leq m\left(r,\frac{f^{(m)}}{f}\right)+m\left(r,\frac{f^{(m-1)}}{f}\right)+\ldots+m\left(r,\frac{f'}{f}\right)\\
&+m\left(r,A_{m-1}\right)+\ldots+m\left(r,A_1\right)+m\left(r,\frac{H}{f}\right)
\end{align*}
\begin{align*}
T(r,A_0)&\leq O(\log{r})+T\left(r,A_{m-1}\right)+\ldots+T\left(r,A_1\right)+T(r,f)+T(r,H)\\
&= O(\log{r})+o\left(T\left(r,A_0\right)\right)+T(r,f)
\end{align*}
for all $r\geq R$ and $r\in S$ where $\overline{\logdens}(S)>0$. Combining the equations, we obtain $\rho(A_0)\leq \rho(f)$. Similarly, when $\rho(A_0)<\rho(A_j)$ then using equation \eqref{sde3} we have
\begin{align*}
|A_j(z)|&\leq \left|\frac{f^{(m)}(z)}{f^{(j)}(z)}\right|+|A_{m-1}(z)|\left|\frac{f^{(m-1)}(z)}{f^{(j)}(z)}\right|+\ldots+ |A_{j+1}(z)|\left|\frac{f^{(j+1)}(z)}{f^{(j)}(z)}\right|\\
& +|A_{j-1}(z)|\left|\frac{f^{(j-1)}(z)}{f^{(j)}(z)}\right|+\ldots+ |A_0(z)|\left|\frac{f(z)}{f^{(j)}(z)}\right|+\left|\frac{H(z)}{f^{(j)}(z)}\right|\\
&= \left|\frac{f^{(m)}(z)}{f^{(j)}(z)}\right|+|A_{m-1}(z)|\left|\frac{f^{(m-1)}(z)}{f^{(j)}(z)}\right|+\ldots+ |A_{j+1}(z)|\left|\frac{f^{(j+1)}(z)}{f^{(j)}(z)}\right|\\
&+\left|\frac{f(z)}{f^{(j)}(z)}\right|\left(|A_{j-1}(z)|\left|\frac{f^{(j-1)}(z)}{f(z)}\right|+\ldots+|A_0(z)|+\left|\frac{H(z)}{f(z)}\right|\right)
\end{align*}
This will imply
\begin{align*}
m(r,A_j)&\leq m\left(r,\frac{f^{(m)}}{f^{(j)}}\right)+ m\left(r,\frac{f^{(m-1)}}{f^{(j)}}\right)+\ldots+ m\left(r,\frac{f^{(j+1)}}{f^{(j)}}\right)+ m\left(r,\frac{f}{f^{(j)}}\right)+\ldots\\
&+ m\left(r,A_{m-1}\right) +\ldots + m\left(r,A_0\right) + m\left(r,\frac{H}{f}\right)
\end{align*}
Now using first fundamental theorem of Nevanlinna theory, lemma of logarithmic derivatives, Lemma \ref{lemm} and \ref{olem} we obtain
\begin{align*}
T\left(r,A_j \right)&\leq O\left(\log{r}\right)+o(1)+\sum_{k=0, k\neq j}^{m-1}T\left(r,  A_k \right)+T\left(r, H \right)+T\left(r, f\right)\\
&=O\left(\log{r}\right)+o(1)+o\left( T\left(r, A_j\right)\right)+T(r,f)
\end{align*}
for sufficiently large $r\in S\setminus F$. This will imply that $\rho(A_j)\leq\rho(f)$.
\end{proof}
It is to be noted that hypothesis of Lemma \ref{probig} are only necessary and not sufficient. Examples \ref{higeg1}, \ref{higeg2} and \ref{higeg7} justifies that hypothesis of Lemma \ref{probig} are not sufficient. Also, Examples \ref{higeg4} - \ref{higeg6} justifies that hypothesis of Lemma \ref{probig} are necessary.

\begin{proof}[{\bf{Proof of Theorem \ref{higthm}}}]
\begin{enumerate}[(a)]
\item Suppose there is a transcendental solution $f$ of equation \eqref{sde3} having finite order. From Lemma \ref{gunlem}, there exists a set $E\subset (1,\infty)$ satisfying $m_l(E)<\infty$ such that 
\begin{equation}\label{hig1}
\left|\frac{f^{(k)}(z)}{f^{(l)}(z)}\right|\leq |z|^{m\rho(f)}, l<k=1,2,\ldots, m-1
\end{equation}
for all $z$ satisfying $|z|=r\notin E\cup[0,1]$ and $|z|\geq R$. Then Lemma \ref{ordlem} implies that there exists $S_1\subset(1,\infty)$ satisfying $0<\underline{\log dens}(S_1)=\delta$ such that 
\begin{equation}\label{hig2}
M(r,A_0)\geq \exp\left(r^{\rho(A_0)-\epsilon}\right)
\end{equation}
for all $r\in S_1$ and $r>R$. We suppose that $\left|f(re^{\iota \theta_r})\right|=M(r,f)$ for each $r$. From Lemma \ref{wale}, for $\delta>0$ and $C\in (0,1)$, there exists $l_0\in (0,1/2)$ and $S_2\subset (1,\infty)$ with $\underline{\logdens}(S_2)\geq 1-\delta/2 $ such that
\begin{equation*}
e^{-5\pi}M(r,f)^{(1-C)}\leq |f(re^{\iota \theta})|
\end{equation*}
for all sufficiently large $r\in S_2$ and $\theta$ such that $|\theta-\theta_r|\leq l_0$. Using Lemma \ref{probig}, $\rho(f)\geq\max\{\rho(A_j),\rho(A_0)\}>\rho(H)$ and hence Lemma \ref{lessord} implies 
\begin{equation}\label{hig3}
\frac{|H(z)|}{M(r,f)}\to 0
\end{equation}
as $r \to \infty$ where $r\in S_3\subset(1,\infty)$ and $\overline{\logdens}(S_3)=1$. 
We know that
$$ \chi_{S_1\cap S_2}=\chi_{S_1}+\chi_{S_2}-\chi_{S_1\cup S_2}$$
and $\overline{\log dens}(S_1\cup S_2)\leq 1$
therefore,
\begin{align*}
\underline{\log dens}(S_1\cap S_2)&\geq \underline{\log dens}(S_1)+\underline{\log dens}(S_2)-\overline{\log dens}(S_1\cup S_2)\\
&\geq \delta +1-\frac{\delta}{2}-1=\frac{\delta}{2}.
\end{align*}
Also,
\begin{align*}
\overline{\log dens}(S_1\cap S_2\cap S_3)&\geq \underline{\log dens}(S_1\cap S_2)+\overline{\log dens}(S_3)\\
&-\overline{\log dens}(S_1\cup S_2\cup S_3)\\
&\geq \frac{\delta}{2}+1-1=\frac{\delta}{2}>0.
\end{align*}
As $m_l(E)<\infty,$ this gives $\overline{\log dens}(S_1\cap S_2\cap S_3 \setminus E)>0$. Hence we can choose $z_q=$\\
$r_q e^{\iota \theta_q}$ with $r_q\rightarrow \infty$ such that 
$$ r_q\in (S_1\cap S_2\cap S_3)\setminus E, \quad \left|f(r_qe^{\iota \theta_q})\right|=M(r_q,f).$$
We may suppose that there exists a subsequence $(\theta_q)$ such that 
$$\lim_{q\rightarrow \infty}\theta_q = \theta_0.$$ 
We have $\lambda(A_j)<\rho(A_j)$ therefore, $A_j(z)=v(z)e^{P(z)}$, where $\rho(v)<\rho\left(e^{P(z)}\right)=n\in \mathbb{N}$. First we consider $\rho(A_j)<\rho(A_0)$ and discuss the following cases:
\begin{enumerate}[(i)]
\item \label{dela} if $\delta(P,\theta_0)>0$, then since $\delta(P,\theta)$ is a continuous function we have,
\begin{equation}
\frac{1}{2}\delta(P,\theta_0)<\delta(P,\theta_m)<\frac{3}{2}\delta(P,\theta_0)
\end{equation}
for all sufficiently large $m\in \mathbb{N}$. From part (i) of Lemma \ref{implem} we have
\begin{equation}\label{hig4}
\exp\left((1-\epsilon)\frac{1}{2}\delta(P,\theta_0)r_q^n\right)\leq |A_j(z_q)|\leq \exp\left( (1+\epsilon)\frac{3}{2}\delta(P,\theta_0)r_q^n\right)
\end{equation}
for sufficiently large $m\in \mathbb{N}.$ Using equations \eqref{sde3}, \eqref{hig1}, \eqref{hig2}, \eqref{hig3} and \eqref{hig4} we get
\begin{equation*}
\begin{split}
\exp\left(r_q^{\rho(A_0)-\epsilon}\right)&\leq M(r,A_0)\\
&\leq \left|\frac{f^{(m)}(z_q)}{f(z_q)}\right|+|A_{m-1}(z_q)|\left|\frac{f^{(m-1)}(z_q)}{f(z_q)}\right|+\ldots+|A_1(z_q)|\left|\frac{f'(z_q)}{f(z_q)}\right|+\left| \frac{H(z_q)}{f(z_q)}\right|\\
&\leq r_q^{m\rho(f)}\left( 1+|A_{m-1}(z_q)|+\ldots+|A_j(z_q)|+\ldots+|A_1(z_q)|\right)+o(1)\\
&\leq  r_q^{m\rho(f)}\left( 1+\exp\left( (1+\epsilon)\frac{3}{2}\delta(P,\theta_0)r_q^n\right) +(m-2)\exp{r_q^{\eta}}\right)+o(1)
\end{split}
\end{equation*}
where $\max\{\rho(A_k): k=1,2,\ldots, m-1, k\neq j\}<\eta<\rho(A_0)$. But this is a contradiction for sufficiently large $r_q$, as $\rho(A_0)>\rho(A_j)=n$.
\item\label{delb} If $\delta(P,\theta_0)<0$, then since $\delta(P,\theta)$ is a continuous function therefore,
\begin{equation*}
\frac{3}{2}\delta(P,\theta_0)<\delta(P,\theta_q)<\frac{1}{2}\delta(P,\theta_0)
\end{equation*}
for sufficiently large $q \in \mathbb{N}$. Using part (ii) of Lemma \ref{implem}, we get
\begin{equation}\label{hig5}
\exp{\left((1+\epsilon)\frac{3}{2}\delta(P,\theta_0)r^n_q\right)}\leq |A(z_q)|\leq \exp{\left((1-\epsilon)\frac{1}{2}\delta(P,\theta_0)r^n_q\right)}
\end{equation}
for sufficiently large $q\in \mathbb{N}$. From equations \eqref{sde3}, \eqref{hig1}, \eqref{hig2}, \eqref{hig3} and \eqref{hig5} we have
\begin{align*}
\exp\left(r_q^{\rho(A_0)-\epsilon}\right)&\leq M(r,A_0)\\
&\leq \left|\frac{f^{(m)}(z_q)}{f(z_q)}\right|+|A_{m-1}(z_q)|\left|\frac{f^{(m-1)}(z_q)}{f(z_q)}\right|+\ldots+|A_1(z_q)|\left|\frac{f'(z_q)}{f(z_q)}\right|+\left| \frac{H(z_q)}{f(z_q)}\right|\\
&\leq r_q^{m\rho(f)}\left( 1+|A_{m-1}(z_q)|+\ldots+|A_j(z_q)|+\ldots+|A_1(z_q)|\right)+o(1)\\
&\leq  r_q^{m\rho(f)}\left( 1+\exp\left( (1-\epsilon)\frac{1}{2}\delta(P,\theta_0)r_q^n\right)+(m-2)\exp{r_q^{\eta}}\right) +o(1)\\
&\leq  r_q^{m\rho(f)}\left( 1+o(1)+(m-2)\exp{r_q^{\eta}}\right)+o(1)
\end{align*}
which will be a contradiction to the fact that $\rho(A_0)>1$.
\item Finally, suppose $\delta(P,\theta_0)=0$. We know that $|\theta_q-\theta_0|\leq l_0$ for sufficiently large $q\in \mathbb{N}$. Choose $\theta_q^*$ such that $l_0/3\leq \theta_q^*-\theta_q\leq l_0$ and $\theta_q^*\to \theta_0^*$ as $q\to \infty$, we have
\begin{equation*}
\theta_q+\frac{l_0}{3}\leq\theta_q^*\leq\theta_q+l_0 \\
\qquad \mbox{which implies}\\
 \qquad \theta_0+\frac{l_0}{3}\leq\theta_0^*\leq\theta_0+l_0
\end{equation*}
as $q\to \infty$. We may assume without loss of generality that $\delta(P,\theta_0^*)>0$ then as done in case (\ref{dela}), we obtain
\begin{equation}\label{hig6}
\exp{\left((1-\epsilon)\frac{1}{2}\delta(P,\theta_0^*)r^n_q\right)}\leq |A_j(z_q^*)|\leq \exp{\left((1+\epsilon)\frac{3}{2}\delta(P,\theta_0^*)r^n_q\right)}.
\end{equation}
for sufficiently large $q\in \mathbb{N}$. Using equations \eqref{sde3}, \eqref{hig1}, \eqref{hig2}, \eqref{hig3} and \eqref{hig6} we get a contradiction as in case (\ref{dela}). Similarly if $\delta(P,\theta_0^*)<0$ then we get contradiction as in case (\ref{delb}).
\end{enumerate}
Now consider $\rho(A_0)<\rho(A_j)$ and following cases:
\begin{enumerate}[(I)]
\item if $\delta(P,\theta_0)>0$ then using equation \eqref{sde3}, \eqref{prelim2}, \eqref{hig1}, \eqref{hig3} and \eqref{hig4} we have
\begin{align*}
\qquad \exp\left( (1-\epsilon) \frac{1}{2}\delta(P,\theta_0)r_q^n\right)&\leq|A_j(z_q)|\\
&\leq \left| \frac{f^{(m)}(z_q)}{f^{(j)}(z_q)}\right|+|A_{m-1}(z_q)|\left|\frac{f^{(m-1)}(z_q)}{f^{(j)}(z_q)}\right| +\ldots\\
&+|A_{j-1}(z_q)|\left|\frac{f^{(j-1)}(z_q)}{f^{(j)}(z_q)}\right|+|A_{j+1}(z_q)|\left|\frac{f^{(j+1)}(z_q)}{f^{(j)}(z_q)}\right|\\
&+\ldots+|A_0(z_q)|\left|\frac{f(z_q)}{f^{(j)}(z_q)}\right|+\left|\frac{H(z_q)}{f^{(j)}(z_q)}\right|\\
&\leq \left| \frac{f^{(m)}(z_q)}{f^{(j)}(z_q)}\right|+|A_{m-1}(z_q)|\left|\frac{f^{(m-1)}(z_q)}{f^{(j)}(z_q)}\right|+\ldots\\
&+\left|\frac{f(z_q)}{f^{(j)}(z_q)}\right|\{|A_{j-1}(z_q)|\left|\frac{f^{(j-1)}(z_q)}{f(z_q)}\right|\\
&+|A_{j+1}(z_q)|\left|\frac{f^{(j+1)}(z_q)}{f(z_q)}\right|
+\ldots+|A_0(z_q)|+\left|\frac{H(z_q)}{f(z_q)}\right|\}\\
&\leq r_q^{m\rho(f)}+|A_{m-1}(z_q)|r_q^{m\rho(f)}+\ldots+r_q^{m\rho(f)}(1+o(1))\\
&\left( |A_{j-1}| r_q^{m\rho(f)}+ |A_{j+1}| r_q^{m\rho(f)}+\ldots+ |A_0(z_q)|+o(1)\right)\\
&\leq r_q^{2m\rho(f)}(1+o(1))\left((m-1) \exp\left( r^\eta\right)+o(1)\right)
\end{align*}
where $\max\{\rho(A_k):k=1,2,\ldots,m-1, k\neq j\}<\rho(A_0)<\eta<\rho(A_j)$. But this gives a contradiction to the fact that $\rho(A_j)>\rho(A_0)$.
\item  When $\delta(P,\theta_0)<0$ or $\delta(P,\theta_0)=0$, then as done in earlier cases, we obtain a contradiction.
\end{enumerate}
Thus all solutions of equation \eqref{sde3} are of infinite order. 
\item  Now, from equation \eqref{sde3} we have 
\begin{equation*}
\frac{1}{f}=-\frac{1}{H}\left( \frac{f^{(m)}}{f}+A_{m-1} \frac{f^{(m-1)}}{f}+\ldots+A_1 \frac{f'}{f}+A_0 \right)
\end{equation*}
As a consequence of lemma of logarithmic derivatives and first fundamental theorem of Nevanlinna theory we have
\begin{align*}
m\left(r,\frac{1}{f}\right)&\leq m\left(r,\frac{f^{(m)}}{f}\right)+\ldots+m\left(r,\frac{f'}{f}\right)+m\left(r,A_{m-1}\right)+\ldots\\
&+ m\left(r,A_1\right)+m\left(r,A_0\right)+m\left(r,\frac{1}{H}\right)\\
&\leq S(r,f)+o(T(r,f))+m(r,H)+O(1)\\
&= S(r,f)+o(T(r,f))+O(1)
\end{align*}
Again applying first fundamental theorem of Nevanlinna theory, we get
\begin{align*}
T(r,f)+O(1)&=m\left(r,\frac{1}{f}\right)+N\left(r,\frac{1}{f}\right)\\
&\leq S(r,f)+N\left(r,\frac{1}{f}\right)+o(T(r,f))+O(1)
\end{align*}
From here, it is easy to conclude that $\lambda(f)=\infty.$
\item Using Lemma \ref{gunlem}, there exists $E\subset(1,\infty)$ satisfying $m_l(E)<\infty$ such that 
\begin{equation}\label{hig7}
\left| \frac{f^{(l)}(z)}{f^{(p)}(z)}\right|\leq c [T(2r,f)]^{2(l-p)}
\end{equation}
where $p<l$ are non-negative integers, $c>0$ is a constant and $z$ satisfies $|z|=r\notin E\cup[0,1]$. Let us suppose that $\rho(A_j)<\rho(A_0)$. Then as in case (a), using equations \eqref{sde3},  \eqref{hig2}, \eqref{hig3} and \eqref{hig7} we get
\begin{align*}
\exp\left(r_q^{\rho(A_0)-\epsilon}\right)&\leq M(r,A_0)\\
&\leq \left|\frac{f^{(m)}(z_q)}{f(z_q)}\right|+|A_{m-1}(z_q)|\left|\frac{f^{(m-1)}(z_q)}{f(z_q)}\right|+\ldots+|A_1(z_q)|\left|\frac{f'(z_q)}{f(z_q)}\right|+\left| \frac{H(z_q)}{f(z_q)}\right|\\
&\leq c[T(2r_q,f)]^{m\rho(f)}\left( 1+|A_{m-1}(z_q)|+\ldots+|A_j(z_q)|+\ldots+|A_1(z_q)|\right)+o(1)\\
&\leq  c[T(2r_q,f)]^{m\rho(f)}\left( 1+(m-1)\exp{r_q^{\eta}}\right)+o(1)
\end{align*}
where $\rho(A_k)<\eta<\rho(A_0)$ for all $k=1,2,\ldots, m-1$. This will imply that $\rho_2(f)\geq \rho(A_0)$.

Now, if $\rho(A_0)<\rho(A_j)$ then as done in case (a), using equations \eqref{sde3}, \eqref{prelim2}, \eqref{hig2}, \eqref{hig3} and \eqref{hig7} we conclude from here that $\rho_2(f)\geq \rho(A_j).$

We know that if $f$ is a solution of equation \eqref{sde3} then 
\begin{equation}\label{hig8}
f(z)=c_1(z)f_1(z)+\ldots+c_m(z)f_m(z)
\end{equation}
where $f_1,\ldots,f_m$ are linearly independent solutions of equation \eqref{sde2} and $c_i'=\frac{HG_i(f_1,f_2,\ldots,f_m)}{W(f_1,f_2,\ldots,f_m)}$ with $G_i(f_1,f_2,\ldots,f_m)$ being a polynomial in $f_1,f_2,\ldots,f_m$ and their derivatives and $W(f_1,f_2,\ldots,f_m)$ being Wronskian of $f_1,f_2,\ldots,f_m$. From equation \eqref{hig8} we obtain
\begin{equation}\label{hig9}
T(r,f)\leq d_1T(r,f_1)+d_2T(r,f_2)+\ldots+d_mT(r,f_m)+dT(r,H)+O(1)
\end{equation}
where $d,d_1,d_2,\ldots,d_m$ are positive integers. From equation \eqref{hig9} and Lemma \ref{asslem} we conclude that $\rho_2(f)\leq \rho=\max\{\rho(A_0),\rho(A_j)\}$.
\item For every complex number $c$, we know that $f\equiv c$ is not a solution of \eqref{sde3} therefore, using Proposition \ref{mohon'ko} and Lemma \ref{olem} we have
\begin{equation*}
m\left(r,\frac{1}{f-c}\right)=S(r,f)
\end{equation*}
for $r\in S$. Thus 
$$\delta(c,f)=\underline{\lim}_{r\to \infty}\frac{m\left(r,\frac{1}{f-c}\right)}{T(r,f)}=0.$$
Therefore, $f$ has no finite deficient value.
\end{enumerate}
\end{proof}

\end{document}